\documentclass[11pt]{article}
\usepackage[a4paper,left=27mm,right=27mm,top=24mm,bottom=24mm,footskip=25pt]{geometry}
\usepackage[T1]{fontenc}
\usepackage{amsmath,amssymb,amsthm,mathtools}
\usepackage{lmodern}
\usepackage{microtype}
\usepackage[hypertexnames=false,hidelinks]{hyperref}
\numberwithin{equation}{section}
\newtheorem{theorem}{Theorem}[section]
\newtheorem{lemma}[theorem]{Lemma}
\newtheorem{proposition}[theorem]{Proposition}
\newtheorem{corollary}[theorem]{Corollary}
\theoremstyle{remark}
\newtheorem*{remark}{Remark}
\newcommand{\R}{\mathbb R}
\newcommand{\Sn}{\mathbb S^n}
\newcommand{\la}{\lambda}
\newcommand{\ip}[2]{\langle #1,#2\rangle}
\newcommand{\sumall}{\sum_{i=1}^n}
\newcommand{\B}{\sigma_{k-2}(\la\mid i j)}
\allowdisplaybreaks[2]

\usepackage{needspace}
\usepackage{etoolbox}
\usepackage{titlesec}
\titlespacing*{\section}{0pt}{2.3ex plus .5ex minus .2ex}{1.2ex plus .2ex}
\titlespacing*{\subsection}{0pt}{1.9ex plus .4ex minus .2ex}{.8ex plus .2ex}
\titlespacing*{\subsubsection}{0pt}{1.6ex plus .3ex minus .2ex}{.7ex plus .2ex}
\pretocmd{\section}{\Needspace{6\baselineskip}}{}{}
\pretocmd{\subsection}{\Needspace{5\baselineskip}}{}{}
\pretocmd{\subsubsection}{\Needspace{4\baselineskip}}{}{}
\hypersetup{pdftitle={Global Curvature Estimates for k-Convex Hypersurfaces},pdfauthor={Fengrui Yang},pdfsubject={Global curvature estimates for prescribed curvature equations}}
\theoremstyle{plain}
\newtheorem*{maintheorem}{Theorem}

\newcommand{\GRW}{\cite{GRW}}
\newcommand{\SX}{\cite{SX}}
\newcommand{\RWone}{\cite{RWone}}
\newcommand{\RWtwo}{\cite{RWtwo}}
\newcommand{\LT}{\cite{LT}}
\newcommand{\Yang}{\cite{Yang}}
\newcommand{\Zhang}{\cite{Zhang}}
\newcommand{\Yan}{\cite{Yan}}
\newcommand{\Garding}{\cite{Garding}}
\newcommand{\LuSemi}{\cite{LuSemi}}
\newcommand{\Chu}{\cite{Chu}}

\begin{document}
\pagenumbering{arabic}
\begin{center}
{\LARGE Global Curvature Estimates for $k$-Convex Hypersurfaces\par}
\vspace{9pt}
{\large Fengrui Yang\par}
\end{center}
\vspace{7pt}
\begin{center}\textbf{Abstract}\end{center}
\begingroup
\small
\setlength{\parindent}{0pt}
\setlength{\parskip}{5pt}
Establishing global curvature estimates for $k$-convex solutions of prescribed curvature equations is a longstanding problem in fully nonlinear partial differential equations and geometric analysis.

We develop a new approach that combines fractional-linear transformations with the logarithmic concavity of hyperbolic polynomials. The key step is a coercive estimate for the quadratic forms arising from the third-order terms in the maximum-principle argument.

Together with a complementary concavity inequality, this estimate yields global curvature bounds for closed, strictly star-shaped $k$-convex hypersurfaces in $\mathbb R^{n+1}$ satisfying $\sigma_k(\kappa)=f(X,\nu)>0$, throughout the range $3\le k<n<2k$.

\par\endgroup
\vspace{5pt}
\section*{Introduction}
\noindent
This paper addresses the longstanding problem of obtaining global $C^2$ estimates for $k$-convex solutions of the prescribed Weingarten curvature equation
\[
\sigma_k(\kappa(X))=f(X,\nu(X)),\qquad X\in M\subset\mathbb R^{n+1}.
\tag{0.1}\label{eq:intro-curvature}
\]
Here $\kappa=(\kappa_1,\ldots,\kappa_n)$ denotes the principal curvatures of $M$, $\nu$ its outward unit normal, and $\sigma_k$ the $k$-th elementary symmetric polynomial. A hypersurface is called $k$-convex if $\kappa(X)\in\Gamma_k$ at every point, where $\Gamma_k=\{\lambda\in\mathbb R^n:\sigma_j(\lambda)>0,\ 1\leq j\leq k\}$. The choices $k=1,2,n$ correspond, up to normalization, to mean curvature, scalar curvature, and Gauss curvature, respectively.

Equation \eqref{eq:intro-curvature} provides a common analytic framework for several central problems in geometry. These include the Minkowski problem \cite{CY,Nirenberg1953,Pogorelov1953,Pogorelov1978}; Alexandrov's problem of prescribing Weingarten curvature as a function of the outer normal \cite{Alexandrov1956,GG}; the prescription of curvature measures in convex geometry \cite{Alexandrov1942,GLL,GLM,Pogorelov1953}; the construction of hypersurfaces with prescribed curvature \cite{BK,CNS,TW}; the prescribed $L_p$ curvature problem \cite{HI}; and the $L_p$ dual Minkowski problem \cite{BF}, the latter corresponding to $k=n$.

These problems correspond to particular choices of the right-hand side $f(X,\nu)$, within the broader class of geometric equations systematically investigated by Alexandrov \cite{Alexandrov1956}. When $k=1$, equation \eqref{eq:intro-curvature} is quasilinear, and $C^2$ estimates follow from classical theory. When $k=n$, it is of Monge--Amp\`ere type; estimates allowing general dependence on $X$ and $\nu$ were established in the foundational work of Caffarelli, Nirenberg, and Spruck \cite{CNSI}. For the intermediate cases $1<k<n$, important early results addressed special forms of $f$. Caffarelli, Nirenberg, and Spruck \cite{CNS} treated $f=f(X)$ for a broader class of curvature operators, including $\sigma_k$ and curvature quotients $\sigma_k/\sigma_l$. Guan and Guan \cite{GG} obtained estimates for convex hypersurfaces when $f=f(\nu)$. For prescribed curvature measures, Guan, Lin, and Ma \cite{GLM} and Guan, Junfang Li, and Yanyan Li \cite{GLL} made substantial contributions by establishing curvature estimates for the special form $f(X,\nu)=u\,\widetilde f(X)$, where $u=\langle X,\nu\rangle$ is the support function. The author subsequently solved the corresponding prescribed curvature measure problem for star-shaped $k$-convex hypersurfaces in hyperbolic space, establishing existence and uniqueness for $3\le k<n$ \Yang.

\noindent
The geometric scope of \eqref{eq:intro-curvature} naturally leads to the study of the equation with a general positive right-hand side $f(X,\nu)$, allowing simultaneous dependence on position and normal. A central question is whether global curvature estimates can be obtained under the natural $k$-convexity assumption.

Guan, Ren, and Wang \GRW\ made a major advance by establishing global curvature estimates for closed convex hypersurfaces and, when $k=2$, for closed star-shaped $2$-convex hypersurfaces. Spruck and Xiao \SX\ gave an elegant and inspiring proof for the scalar-curvature case, which also applies in space forms. A simpler proof of the estimate in the convex case was subsequently given by Chu \Chu.

For higher curvature orders, Ren and Wang established global curvature estimates for closed star-shaped $k$-convex hypersurfaces in the cases $k=n-1$ \RWone\ and $k=n-2$ \RWtwo. These results allow general positive $f(X,\nu)$ without imposing convexity beyond $k$-convexity. Lu and Tsai \LT\ later gave a simpler proof of the $k=n-1$ estimate. Recently, Jin Yan \Yan\ presented a proof of global curvature estimates for closed star-shaped $k$-convex hypersurfaces in the range $n/2\le k<n$. The results in this paper were obtained independently of Yan's work, and the proof of the key quadratic-form estimate is completely different.

Another line of progress concerns semi-convex solutions, for which the principal curvatures satisfy a prescribed lower bound $\kappa_i\ge -K$. Lu \LuSemi\ established global curvature estimates for closed, strictly star-shaped, semi-convex $k$-convex hypersurfaces in hyperbolic space, allowing general positive $f(X,\nu)$. In Euclidean space, Zhang \Zhang\ introduced a concavity inequality that yields a new proof of the semi-convex global curvature estimate noted by Guan, Ren, and Wang \GRW. These estimates depend on the prescribed semi-convexity bound $K$.

\par\addvspace{.7\baselineskip}

In this paper, we establish global curvature estimates for general positive $f(X,\nu)$ throughout the range
\[
3\le k<n<2k.
\]
Our proof uses the classical maximum principle together with a new construction adapted to the quadratic forms arising in the largest-principal-curvature calculation. We state the closed-hypersurface case here; Theorem~2.1 also gives the corresponding estimate in terms of the boundary curvature when $M$ has boundary.

\begin{maintheorem}[Global curvature estimate]
Let $3\le k<n<2k$, and let $M\subset\R^{n+1}$ be a closed $C^4$, strictly star-shaped, $k$-convex hypersurface satisfying
\[
\sigma_k(\lambda(X))=f(X,\nu(X)).
\]
Suppose that $f\in C^2(U)$ is positive on an open neighborhood $U$ of the unit normal bundle of $M$, and that
\[
|X|\le R,\quad \langle X,\nu\rangle\ge u_*>0,\quad
0<M_1\le f\le M_2,\quad \|f\|_{C^2(U)}\le L.
\]
Then
\[
\max_{X\in M}\max_{1\le i\le n}|\lambda_i(X)|
\le C(n,k,R,u_*,M_1,M_2,L).
\]
\end{maintheorem}

\noindent
The central idea is to design a change of eigenvalue variables that captures the full quadratic form arising in the curvature estimate. The novelty lies in encoding its Hessian and diagonal contributions in a single concavity structure, since their combined control does not follow directly from the standard concavity of $\sigma_k^{1/k}$ and $(\sigma_k/\sigma_l)^{1/(k-l)}$ alone.

To achieve this, we carefully choose coordinatewise fractional-linear transformations so that both contributions arise in a single second-derivative calculation. Clearing denominators and homogenizing then produces a hyperbolic polynomial. Its logarithmic concavity \Garding\ is used along transformed tangent directions, where the first-derivative square of the composed curvature function vanishes.

For this logarithmic concavity to yield the required estimate, the transformations must also retain a quantitative margin. Choosing $\gamma>1$ for the largest coordinate and $\beta<1$ for the others leaves strictly positive quadratic remainders in every coordinate; the condition $n<2k$ makes these choices compatible with the required polynomial positivity. These remainders absorb the denominator errors and yield the needed coercivity. Together with a complementary concavity inequality \Zhang, this completes the proof.

Section~1 records the preliminary identities. Section~2 states the main theorem and derives the maximum-principle inequality. Section~3 develops the fractional-linear transformation and the quantitative estimate, and Section~4 treats the complementary regime. Section~5 handles repeated largest principal curvatures and completes the proof. The polynomial facts used in Section~3 are collected in Appendix~A.

The author is deeply grateful to Professor Pengfei Guan for his help and encouragement over the years.

\section{Preliminaries}

Let $M\subset\R^{n+1}$ be an oriented hypersurface with position vector $X$,
unit normal $\nu$, induced metric $g$, and second fundamental form $h$.
We use the convention
\[
 D_{e_i}\nu=\sum_{j=1}^n h_{ij}e_j,
 \qquad D_{e_i}e_j=\nabla_{e_i}e_j-h_{ij}\nu,
\]
where $D$ and $\nabla$ are the Euclidean and induced connections, respectively.
Thus a sphere has positive principal curvatures with respect to its outer
normal. We write $h_{ij\ell}=(\nabla_{e_\ell}h)_{ij}$ and
$h_{ij\ell r}=(\nabla_{e_r}\nabla_{e_\ell}h)_{ij}$, with the usual covariant
interpretation of the latter notation. Let
$\la=(\la_1,\ldots,\la_n)$ be the principal curvature vector, ordered so that
$\la_1\ge\cdots\ge\la_n$. We set
\[
 \Gamma_k=\{\la\in\R^n:\sigma_j(\la)>0,\ 1\le j\le k\},
\]
and call $M$ $k$-convex if $\la(X)\in\Gamma_k$ for every $X\in M$.
Here $\sigma_j$ is the unnormalized $j$th elementary symmetric polynomial;
we set $\sigma_0=1$, and $\sigma_j=0$ when the degree exceeds the number of variables;
$\sigma_j(\la\mid i)$ and $\sigma_j(\la\mid ij)$ denote deletion of the indicated
components.

Regard $F(h)=\sigma_k(\la(h))$ as a function of the second fundamental form,
and write
\[
 F^{ij}=\sigma_k^{ij}=\frac{\partial F}{\partial h_{ij}},
 \qquad
 F^{ij,pq}=\sigma_k^{ij,pq}
       =\frac{\partial^2F}{\partial h_{ij}\partial h_{pq}}.
\]
In an orthonormal frame in which $h$ is diagonal,
\begin{equation}\label{eq:first-F}
 F^{ij}=\sigma_{k-1}(\la\mid i)\delta_{ij},
 \qquad F^i:=F^{ii}=\sigma_{k-1}(\la\mid i).
\end{equation}
The only possibly nonzero second derivatives are
\begin{equation}\label{eq:second-F}
 \begin{aligned}
 F^{ii,jj}&=\sigma_{k-2}(\la\mid ij),&&i\ne j,\\
 F^{ij,ji}&=-\sigma_{k-2}(\la\mid ij),&&i\ne j.
 \end{aligned}
\end{equation}
We use the full-index convention, so for every symmetric matrix $\eta$,
\begin{equation}\label{eq:F-contraction}
 \sum_{i,j,p,q=1}^n F^{ij,pq}\eta_{ij}\eta_{pq}
 =\sum_{i\ne j}\B\eta_{ii}\eta_{jj}
   -\sum_{i\ne j}\B\eta_{ij}^2.
\end{equation}
In particular, both sums over $i\ne j$ are ordered sums. We will also use
\begin{equation}\label{eq:F-identities}
 \sumall F^i\la_i=kF,
 \qquad
 F^i-F^j=(\la_j-\la_i)\sigma_{k-2}(\la\mid ij)
 \quad (i\ne j).
\end{equation}
These identities hold without any assumption on the multiplicities of the
principal curvatures.

For $\la\in\Gamma_k$, the standard cone properties give
\begin{equation}\label{eq:cone}
 F^i>0,\qquad \sigma_{k-2}(\la\mid ij)>0\quad(i\ne j),
 \qquad \la_n\ge-\frac{n-k}{k}\la_1;
\end{equation}
see, for example, \cite[Lemma 2.2]{LT}. Consequently, if $n\le2k-1$,
\begin{equation}\label{eq:positive-sum}
 \la_1+\la_i\ge\frac{2k-n}{k}\la_1
 \ge\frac{\la_1}{k}>0\quad(1\le i\le n).
\end{equation}
In particular, $|\la_i|\le\la_1$ in the dimension range considered below.

We also record the standard lower bound
\begin{equation}\label{eq:cone-lower-bound}
 \la_1\sigma_{k-1}(\la\mid1)\ge\frac{k}{n}\sigma_k(\la).
\end{equation}
See \cite[Lemma 2.2]{LT}.
The concavity of $\sigma_k^{1/k}$ on $\Gamma_k$ gives
\begin{equation}\label{eq:sigma-concavity}
 D^2\sigma_k(\la)[\xi,\xi]
 \le\frac{k-1}{k\sigma_k(\la)}
          \left(\sumall F^i\xi_i\right)^2,
 \qquad \xi\in\R^n.
\end{equation}
We use throughout the standard deletion positivity
$\sigma_{k-r}(\la\mid I)>0$ for $|I|=r$ and $1\le r\le k$,
as well as the Newton--Maclaurin inequalities.

Let $u=\ip{X}{\nu}$ be the support function. The Codazzi equation and the
support-function identities are
\begin{equation}\label{eq:codazzi-support}
 \begin{aligned}
 h_{ij\ell}&=h_{i\ell j},\\
 u_i&=\sum_{j=1}^n h_{ij}\ip{X}{e_j},\\
 u_{ij}&=\sum_{\ell=1}^n h_{ij\ell}\ip{X}{e_\ell}
           +h_{ij}-u\sum_{\ell=1}^n h_{i\ell}h_{j\ell}.
 \end{aligned}
\end{equation}
At a point where $h$ is diagonal, these reduce to
\begin{equation}\label{eq:u-diagonal}
 u_i=\la_i\ip{X}{e_i},
 \qquad u_{ii}=\sum_{\ell=1}^n h_{ii\ell}\ip{X}{e_\ell}
                     +\la_i-u\la_i^2.
\end{equation}
At the same point, the Euclidean commutation formula reads
\begin{equation}\label{eq:commutation}
 h_{11ii}=h_{ii11}+\la_1^2\la_i-\la_1\la_i^2.
\end{equation}
See also \cite[Section 2]{LT} for these geometric identities.

\section{Curvature estimates}

\begin{theorem}\label{thm:curvature}
Let $k\ge3$ and $k+1\le n\le2k-1$. Let $M\subset\R^{n+1}$ be a compact
$C^4$, strictly star-shaped, $k$-convex hypersurface, possibly with boundary,
satisfying
\begin{equation}\label{eq:curvature-equation}
 \sigma_k(\la(X))=f(X,\nu(X)),\qquad X\in M.
\end{equation}
Suppose that $f\in C^2(\mathcal U)$ is positive, where $\mathcal U$ is an open
neighborhood of the unit normal bundle of $M$ in $\R^{n+1}\times\Sn$.
Assume that the given bounds are
\[
 |X|\le R,\qquad u=\ip{X}{\nu}\ge u_*>0,\qquad
 0<M_1\le f\le M_2,\qquad \|f\|_{C^2(\mathcal U)}\le L.
\]
Then there is a constant $C=C(n,k,R,u_*,M_1,M_2,L)$ such that
\begin{equation}\label{eq:curvature-estimate}
 \max_{X\in M}\max_{1\le i\le n}|\la_i(X)|
 \le C\left(1+\max_{X\in\partial M}\max_{1\le i\le n}|\la_i(X)|\right).
\end{equation}
When $\partial M=\varnothing$, the boundary term is omitted.
\end{theorem}

\noindent\textit{Proof (maximum-principle computation).}
We use the standard maximum-principle framework of
\cite[Theorem 16]{GRW}, \cite[Theorem 4.1]{LT}, and
\cite[Section 4]{Yang}. Consider
\begin{equation}\label{eq:test-function}
 \phi=\log\la_1-N\log u,
\end{equation}
where $N\ge1$ is a sufficiently large constant, depending only on the data in
Theorem~\ref{thm:curvature}, to be fixed in the subsequent estimates.
Since $u_*\le u\le R$, an upper bound for $\la_1$ at a maximum point of $\phi$
yields an upper bound on all of $M$. If the maximum is attained on $\partial M$,
\eqref{eq:curvature-estimate} follows directly from this observation and
\eqref{eq:positive-sum}. We may therefore assume that $\phi$ attains its maximum
at an interior point $x_0$. We may also assume that $\la_1(x_0)\ge1$.

To handle possible multiplicities, define the smooth upper supporting function
\begin{equation}\label{eq:upper-support}
 \psi(x)=e^{\phi(x_0)}u(x)^N
        =\la_1(x_0)\left(\frac{u(x)}{u(x_0)}\right)^N.
\end{equation}
Then $\psi\ge\la_1$ and $\psi(x_0)=\la_1(x_0)$. Moreover,
\begin{equation}\label{eq:constant-test}
 \widetilde\phi:=\log\psi-N\log u\equiv\phi(x_0).
\end{equation}
Choose a local orthonormal frame $e_1,\ldots,e_n$, geodesic at $x_0$, such that
$h_{ij}(x_0)=\la_i\delta_{ij}$. Write
\begin{equation}\label{eq:multiplicity}
 \la_1=\cdots=\la_m>\la_{m+1}\ge\cdots\ge\la_n,
 \qquad 1\le m\le n.
\end{equation}
If $m=n$, the strict comparison with $\la_{m+1}$ is omitted, and all sums
over indices larger than $m$ are understood to be empty.
All quantities and covariant derivatives in the rest of the computation are
evaluated at $x_0$.

The upper-support version of the eigenvalue-contact lemma of
Brendle--Choi--Daskalopoulos \cite[Lemma 5]{BCD} gives
\begin{equation}\label{eq:contact-first}
 h_{pq i}=\delta_{pq}\psi_i,
 \qquad 1\le p,q\le m,\quad 1\le i\le n,
\end{equation}
and
\begin{equation}\label{eq:contact-second}
 \psi_{ii}\ge h_{11ii}
        +2\sum_{p>m}\frac{h_{1pi}^2}{{\la_1}-\la_p},
 \qquad 1\le i\le n.
\end{equation}
The cited lemma is stated for a lower support of the smallest eigenvalue;
applying its tensor argument to $-h$ gives
\eqref{eq:contact-first}--\eqref{eq:contact-second}.
In particular, $\psi_i=h_{11i}$.

Differentiating \eqref{eq:constant-test} once gives the critical equation
\begin{equation}\label{eq:critical}
 0=\frac{\psi_i}{\psi}-N\frac{u_i}{u}
   =\frac{h_{11i}}{\la_1}-N\frac{u_i}{u}
   =\frac{h_{11i}}{\la_1}-\frac{N\la_i\ip{X}{e_i}}{u},
 \qquad 1\le i\le n.
\end{equation}
Differentiating twice, contracting with $F^{ii}=F^i>0$, and using
\eqref{eq:contact-second}, we obtain
\begin{equation}\label{eq:max-second}
 \begin{aligned}
 0=\sumall F^i\widetilde\phi_{ii}
 &\ge \frac{1}{\la_1}\sumall F^i h_{11ii}
   +\frac{2}{\la_1}\sumall\sum_{p>m}\frac{F^i h_{1pi}^2}{{\la_1}-\la_p}
   -\frac1{{\la_1}^2}\sumall F^i h_{11i}^2\\
 &\hspace{8mm}-\frac Nu\sumall F^i u_{ii}
   +\frac N{u^2}\sumall F^i u_i^2.
 \end{aligned}
\end{equation}
This is the rigorous supporting-function interpretation of the
maximum-principle inequality for $\sum_i\sigma_k^{ii}\phi_{ii}$.

By \eqref{eq:commutation} and \eqref{eq:F-identities},
\begin{equation}\label{eq:commuted-contracted}
 \frac{1}{\la_1}\sumall F^i h_{11ii}
 =\frac{1}{\la_1}\sumall F^i h_{ii11}
     +kf{\la_1}-\sumall F^i\la_i^2.
\end{equation}
Write $f_i$ and $f_{11}$ for covariant derivatives on $M$ of the composite
function $X\mapsto f(X,\nu(X))$. Differentiating
\eqref{eq:curvature-equation} once yields
\begin{equation}\label{eq:equation-first}
 \sumall F^i h_{ii\ell}=f_\ell
      =d_Xf(e_\ell)+\la_\ell d_\nu f(e_\ell),
 \qquad 1\le\ell\le n.
\end{equation}
Therefore, \eqref{eq:u-diagonal} gives
\begin{equation}\label{eq:u-contracted}
 \sumall F^i u_{ii}
   =\sum_{\ell=1}^n f_\ell\ip{X}{e_\ell}
        +kf-u\sumall F^i\la_i^2.
\end{equation}
Substituting \eqref{eq:commuted-contracted} and \eqref{eq:u-contracted} into
\eqref{eq:max-second}, we find
\begin{equation}\label{eq:before-f-second}
 \begin{aligned}
 0\ge{}&\frac{1}{\la_1}\sumall F^i h_{ii11}
    +\frac{2}{\la_1}\sumall\sum_{p>m}\frac{F^i h_{1pi}^2}{{\la_1}-\la_p}
    -\frac1{{\la_1}^2}\sumall F^i h_{11i}^2\\
 &+(N-1)\sumall F^i\la_i^2+\frac N{u^2}\sumall F^i u_i^2
    +kf{\la_1}-\frac{Nkf}{u}
    -\frac Nu\sum_{\ell=1}^n f_\ell\ip{X}{e_\ell}.
 \end{aligned}
\end{equation}

Differentiating \eqref{eq:curvature-equation} twice in the $e_1$ direction gives
\begin{equation}\label{eq:equation-second}
 \sumall F^i h_{ii11}
    +\sum_{p,q,r,s=1}^n F^{pq,rs}h_{pq1}h_{rs1}=f_{11}.
\end{equation}
For clarity, all derivatives of $f$ in its two arguments below are taken with
respect to the product connection on $\R^{n+1}\times\Sn$; in particular,
$d_{\nu\nu}^2f$ is the covariant Hessian on the sphere. The chain rule gives
\begin{equation}\label{eq:f-chain-rule}
 \begin{aligned}
 f_{11}={}&\sum_{\ell=1}^n h_{11\ell}d_\nu f(e_\ell)
       +d_{XX}^2f(e_1,e_1)+2{\la_1}\,d_{X\nu}^2f(e_1,e_1)\\
       &+{\la_1}^2d_{\nu\nu}^2f(e_1,e_1)-{\la_1}\,d_Xf(\nu).
 \end{aligned}
\end{equation}
Indeed, along a geodesic through $x_0$ with initial tangent $e_1$,
$D_1X_1=-{\la_1}\nu$, $D_1\nu={\la_1}e_1$, and
$\nabla^{\Sn}_1(D_1\nu)=\sum_\ell h_{11\ell}e_\ell$.
Since ${\la_1}\ge1$, \eqref{eq:f-chain-rule} implies
\begin{equation}\label{eq:f11-bound}
 \frac{f_{11}}{\la_1}
    \ge\sum_{\ell=1}^n\frac{h_{11\ell}}{\la_1} d_\nu f(e_\ell)-C{\la_1}.
\end{equation}
Here and below, $C$ denotes a positive constant depending only on
$n,k,R,u_*,M_1,M_2,L$, which may change from line to line and is independent of
$N$.

The terms involving $d_\nu f$ cancel by the critical equation:
\begin{equation}\label{eq:drift-cancellation}
 \begin{aligned}
 &\sum_{\ell=1}^n\frac{h_{11\ell}}{\la_1} d_\nu f(e_\ell)
       -\frac Nu\sum_{\ell=1}^n f_\ell\ip{X}{e_\ell}\\
 &\quad=\frac Nu\sum_{\ell=1}^n\la_\ell\ip{X}{e_\ell}d_\nu f(e_\ell)
   -\frac Nu\sum_{\ell=1}^n
       \bigl(d_Xf(e_\ell)+\la_\ell d_\nu f(e_\ell)\bigr)\ip{X}{e_\ell}\\
 &\quad=-\frac Nu\sum_{\ell=1}^n d_Xf(e_\ell)\ip{X}{e_\ell}
       \ge-CN.
 \end{aligned}
\end{equation}
Combining \eqref{eq:before-f-second}--\eqref{eq:drift-cancellation}, and using
$kf{\la_1}\ge0$ and $Nkf/u\le CN$, we arrive at
\begin{equation}\tag{$\clubsuit 1$}\label{eq:club1}
 \begin{aligned}
 0\ge{}&-\frac{1}{\la_1}\sum_{p,q,r,s=1}^n F^{pq,rs}h_{pq1}h_{rs1}
       +\frac{2}{\la_1}\sumall\sum_{p>m}\frac{F^i h_{1pi}^2}{{\la_1}-\la_p}
       -\frac1{{\la_1}^2}\sumall F^i h_{11i}^2\\
       &+(N-1)\sumall F^i\la_i^2
         +\frac N{u^2}\sumall F^i u_i^2-C{\la_1}-CN.
 \end{aligned}
\end{equation}

We next combine the positive third-order terms with the negative terms
$-F^i h_{11i}^2/{\la_1}^2$. By \eqref{eq:F-contraction},
\begin{equation}\label{eq:split-Hessian}
 -\sum_{p,q,r,s=1}^n F^{pq,rs}h_{pq1}h_{rs1}
 =-\sum_{i\ne j}\B h_{ii1}h_{jj1}
    +\sum_{i\ne j}\B h_{ij1}^2.
\end{equation}
Using \eqref{eq:cone}, Codazzi, and \eqref{eq:F-identities}, we obtain
\begin{equation}\label{eq:off-diagonal-good}
 \begin{aligned}
 \frac{1}{\la_1}\sum_{i\ne j}\B h_{ij1}^2
 &\ge\frac{2}{\la_1}\sum_{i>m}\sigma_{k-2}(\la\mid1i)h_{11i}^2\\
 &=\frac{2}{\la_1}\sum_{i>m}\frac{F^i-F^1}{{\la_1}-\la_i}h_{11i}^2.
 \end{aligned}
\end{equation}
On the other hand, retaining the distinct terms with $i=p$ and $i=1$ in the
following double sum yields
\begin{equation}\label{eq:eigenvalue-good}
 \begin{aligned}
 \frac{2}{\la_1}\sumall\sum_{p>m}\frac{F^i h_{1pi}^2}{{\la_1}-\la_p}
 &\ge\frac{2}{\la_1}\sum_{p>m}\frac{F^p h_{1pp}^2}{{\la_1}-\la_p}
       +\frac{2}{\la_1}\sum_{p>m}\frac{F^1 h_{1p1}^2}{{\la_1}-\la_p}\\
 &=\frac{2}{\la_1}\sum_{i>m}\frac{F^i h_{ii1}^2}{{\la_1}-\la_i}
       +\frac{2}{\la_1}\sum_{i>m}\frac{F^1 h_{11i}^2}{{\la_1}-\la_i}.
 \end{aligned}
\end{equation}
Furthermore, \eqref{eq:contact-first} and Codazzi imply
\begin{equation}\label{eq:top-block-vanishing}
 h_{11i}=h_{1i1}=0,\qquad 2\le i\le m.
\end{equation}
Adding \eqref{eq:off-diagonal-good} and \eqref{eq:eigenvalue-good}, and then
subtracting ${\la_1}^{-2}\sum_iF^ih_{11i}^2$, gives
\begin{equation}\label{eq:combined-good}
 \begin{aligned}
 &\frac{1}{\la_1}\sum_{i\ne j}\B h_{ij1}^2
   +\frac{2}{\la_1}\sumall\sum_{p>m}\frac{F^i h_{1pi}^2}{{\la_1}-\la_p}
   -\frac1{{\la_1}^2}\sumall F^i h_{11i}^2\\
 &\quad\ge\frac{2}{\la_1}\sum_{i>m}\frac{F^i h_{ii1}^2}{{\la_1}-\la_i}
        -\frac{F^1h_{111}^2}{{\la_1}^2}
        +\sum_{i>m}\frac{F^i({\la_1}+\la_i)}{{\la_1}^2({\la_1}-\la_i)}h_{11i}^2.
 \end{aligned}
\end{equation}
Here the last coefficient follows from the exact identity
\[
 \frac{2F^i}{{\la_1}({\la_1}-\la_i)}-\frac{F^i}{{\la_1}^2}
       =\frac{F^i({\la_1}+\la_i)}{{\la_1}^2({\la_1}-\la_i)}.
\]
Its sign is positive by \eqref{eq:positive-sum}. Thus, substituting
\eqref{eq:split-Hessian} and \eqref{eq:combined-good} into \eqref{eq:club1},
and discarding this nonnegative sum together with
$Nu^{-2}\sum_iF^iu_i^2$, we conclude that
\begin{equation}\tag{$\clubsuit 2$}\label{eq:club2}
 \boxed{\begin{aligned}
 0\ge{}&-\frac{1}{\la_1}\sum_{i\ne j}
       \sigma_{k-2}(\la\mid ij)h_{ii1}h_{jj1}
       +\frac{2}{\la_1}\sum_{i>m}
          \frac{\sigma_{k-1}(\la\mid i)}{\la_1-\la_i}h_{ii1}^2\\
       &-\sigma_{k-1}(\la\mid1)\left(\frac{h_{111}}{\la_1}\right)^2
       +(N-1)\sumall\sigma_{k-1}(\la\mid i)\la_i^2
       -C\la_1-CN.
 \end{aligned}}
\end{equation}

To express the remaining third-order terms in quadratic-form notation, set
$x_i=h_{ii1}$ and define
\begin{equation}\label{eq:L-definition}
 L_\la(x)=-\sum_{1\le i<j\le n}\B x_ix_j
       +\sum_{i>m}\frac{F^i}{{\la_1}-\la_i}x_i^2
       -\frac{F^1}{2{\la_1}}x_1^2.
\end{equation}
Then \eqref{eq:club2} is precisely
\begin{equation}\label{eq:L-reduction}
 0\ge\frac{2}{\la_1}L_\la(x)+(N-1)\sumall F^i\la_i^2-C{\la_1}-CN.
\end{equation}
The vector $x$ satisfies
\begin{equation}\label{eq:affine-constraint}
 x_1=\cdots=x_m,
 \qquad
 \sumall F^i x_i=f_1=d_Xf(e_1)+{\la_1}\,d_\nu f(e_1)=C_1{\la_1},
 \qquad |C_1|\le C,
\end{equation}
where $C_1=d_\nu f(e_1)+{\la_1}^{-1}d_Xf(e_1)$ is a bounded scalar at $x_0$.
Thus the remaining step in the curvature estimate is to control
$L_\la$ under \eqref{eq:affine-constraint}.

The next two sections treat the case $m=1$, with the small-$f$ regime first.
The repeated-eigenvalue case and the completion of the proof are given in
Section~\ref{sec:completion}.

\section{The regime \texorpdfstring{$\sigma_k\leq\varepsilon_I\lambda_1F^1$}{sigma-k <= epsilon-I lambda-1 F-1}}
\label{sm:section}

We first treat the regime in which $\sigma_k$ is small relative to
$\lambda_1F^1$. The argument takes place directly in the interior of
$\Gamma_k$ and uses a fractional-linear transformation followed by
logarithmic concavity. In this section the largest coordinate is assumed
to be simple; the case of a multiple largest coordinate is addressed
below.

\subsection{Notation and the interior estimate}

Throughout this section, assume
\begin{equation}\label{sm:assumptions}
 k\geq3,\qquad k<n\leq2k-1,\qquad
 \lambda_1>\lambda_2\geq\cdots\geq\lambda_n,
 \qquad\lambda\in\Gamma_k.
\end{equation}
Set
\begin{equation}\label{sm:notation}
 f=\sigma_k(\lambda),\qquad
 F^i=\sigma_{k-1}(\lambda\mid i),\qquad
 H=D^2\sigma_k(\lambda),\qquad
 s=\sum_{i=1}^n F^i x_i.
\end{equation}
Here $x\in\mathbb R^n$ is an arbitrary vector, and $H$ is the Hessian
of $\sigma_k$ as a function of the eigenvalue variables. Thus $H$ is
distinct from the second fundamental form. We use deletion positivity
on $\Gamma_k$:
\[
 \sigma_{k-r}(\lambda\mid I)>0
 \quad\text{when }|I|=r,\quad 1\leq r\leq k.
\]
The auxiliary polynomial and convexity results used in the argument are
proved in Appendix~\ref{poly:section}.

Recall the quadratic form
\begin{equation}\label{sm:quadratic}
 L_\lambda(x)
 =-\sum_{i<j}\sigma_{k-2}(\lambda\mid ij)x_ix_j
  +\sum_{i=2}^n\frac{F^i}{\lambda_1-\lambda_i}x_i^2
  -\frac{F^1}{2\lambda_1}x_1^2.
\end{equation}
Its first term is $-\tfrac12x^{\mathsf T}Hx$. The ordering and
$\sigma_1>0$ imply
\begin{equation}\label{sm:coordinate-bounds}
 \lambda_1>0,\qquad
 -(n-1)\lambda_1<\lambda_i\leq\lambda_1,
 \qquad
 F^i-F^1=(\lambda_1-\lambda_i)
 \sigma_{k-2}(\lambda\mid1i)\geq0\quad(i\geq2).
\end{equation}
Define the constants
\begin{equation}\label{sm:constants}
 \Delta=2k-n,\qquad
 \varepsilon_I=\frac{\Delta}{2k+n},\qquad
 \theta=\frac{\Delta}{4kn^2},\qquad
 \mathfrak h=k(n-k).
\end{equation}
They are positive and depend only on $n$ and $k$.

\begin{proposition}[Direct interior estimate]\label{sm:main}
Under \eqref{sm:assumptions}, suppose
\begin{equation}\label{sm:regime}
 0<f\leq E_I(\lambda),\qquad
 E_I(\lambda)=\varepsilon_I\lambda_1F^1.
\end{equation}
Then, for every $x\in\mathbb R^n$,
\begin{equation}\label{sm:main-bound}
 L_\lambda(x)\geq
 -\left(\frac12+\frac{2k^2n^2(n-k)}{2k-n}\right)
 \frac{s^2}{\lambda_1F^1}.
\end{equation}
Consequently, if $\lambda_1F^1\geq c_0>0$ and
$s=C_1\lambda_1$, then
\begin{equation}\label{sm:affine-consequence}
 L_\lambda(x)\geq
 -\frac{C_1^2}{c_0}
 \left(\frac12+\frac{2k^2n^2(n-k)}{2k-n}\right)\lambda_1^2.
\end{equation}
The simpler condition $0<f\leq\varepsilon_Ic_0$, together with
$\lambda_1F^1\geq c_0$, also suffices.
\end{proposition}

\subsection{Mixed coefficients and the affine constraint}

For $i\geq2$, put
\[
 G_i=\sigma_{k-1}(\lambda\mid1i),\qquad
 U_i=\sigma_{k-2}(\lambda\mid1i),\qquad
 V_i=\sigma_k(\lambda\mid1i),\qquad
 \mathcal R_{n,k}=\frac{k(n-k)}{n-1}.
\]

\begin{lemma}[Mixed coefficients]\label{sm:mixed}
Under \eqref{sm:assumptions}, for each $i\geq2$,
\begin{equation}\label{sm:mixed-bound}
 G_i^2\leq\mathcal R_{n,k} F^1F^i.
\end{equation}
\end{lemma}

\begin{proof}
Expansion in the first and $i$th coordinates gives
\[
 f=V_i+(\lambda_1+\lambda_i)G_i+\lambda_1\lambda_iU_i,
 \qquad
 F^1=G_i+\lambda_iU_i,\qquad
 F^i=G_i+\lambda_1U_i.
\]
Consequently,
\begin{equation}\label{sm:mixed-identity}
 F^1F^i=G_i^2-U_iV_i+fU_i.
\end{equation}
When $n=k+1$, one has $V_i=0$ and $\mathcal R_{n,k}=1$, so the result follows
immediately. For $n\geq k+2$,
Newton's inequality applied to $\lambda\mid1i$ gives
\[
 U_iV_i\leq
 \frac{(k-1)(n-k-1)}{k(n-k)}G_i^2.
\]
Combining this with \eqref{sm:mixed-identity} and $fU_i\geq0$ yields
$F^1F^i\geq G_i^2/\mathcal R_{n,k}$.
\end{proof}

\subsection{The fractional-linear transformation}
\label{sm:fractional}

The change of variables is motivated by the two distinct parts of $L_\lambda$.
Since $H_{ii}=0$ and $H_{ij}=\sigma_{k-2}(\lambda\mid ij)$ for $i\ne j$,
we can write
\[
 L_\lambda(x)=-\frac12x^{\mathsf T}Hx-\frac{F^1}{2\lambda_1}x_1^2
 +\sum_{i=2}^n\frac{F^i}{\lambda_1-\lambda_i}x_i^2.
\]
Thus the first part comes from the Hessian of $\sigma_k$, while the second
is a prescribed diagonal quadratic form. Applying logarithmic concavity
directly to $\sigma_k$ controls the Hessian expression on $\sum_iF^iy_i=0$,
but does not by itself produce these diagonal coefficients. The idea is to
generate both parts in one second-derivative calculation by allowing each
eigenvalue coordinate to depend on its own new variable:
\[
 \lambda_i(z_i)=m_i(z_i),\qquad
 \mathcal F(z)=\sigma_k\bigl(m_1(z_1),\ldots,m_n(z_n)\bigr).
\]
Indeed, at a point $q$ with $m(q)=\lambda$ and $m_i'(q_i)\ne0$, take
$v_i=y_i/m_i'(q_i)$. The chain rule gives
\[
 D^2\mathcal F(q)[v,v]=y^{\mathsf T}Hy
 +\sum_{i=1}^nF^i\frac{m_i''(q_i)}{(m_i'(q_i))^2}y_i^2.
\]
The coordinatewise dependence makes the additional term diagonal. Choosing
the ratios $m_i''/(m_i')^2$ appropriately therefore allows the same calculation
to incorporate the diagonal part of $L_\lambda$. The fractional-linear maps
below implement this idea; the required logarithmic concavity will be
established for the polynomial obtained by clearing their denominators.

For the coordinate transformation in this subsection and its associated
calculation, write $a=\lambda_1$. Fix the eigenvalue vector and hold
$a$ constant during all differentiation. Choose
\begin{equation}\label{sm:parameters}
 \gamma=\frac{n}{2(n-k)}>1,\qquad
 \beta=\frac{2k+n}{4k}<1,
 \qquad \beta>\gamma\frac{n-k}{k}.
\end{equation}
Define
\begin{equation}\label{sm:maps}
 m_1(z_1)=\frac{az_1}{a-\gamma z_1},\qquad
 m_i(z_i)=\frac{az_i}{a+\beta z_i}\quad(i\geq2),
\end{equation}
and set
\begin{equation}\label{sm:q}
 q_1=\frac{a}{1+\gamma},\qquad
 q_i=\frac{a\lambda_i}{a-\beta\lambda_i}\quad(i\geq2).
\end{equation}
Here $z=(z_1,\ldots,z_n)$ is the new variable, and $q=(q_1,\ldots,q_n)$
is the fixed point in the $z$-space corresponding to the given eigenvalue
vector $\lambda$. Its coordinates are obtained by solving $m_i(q_i)=\lambda_i$.
Thus differentiation is performed in $z$ and then evaluated at $z=q$, with
$a$ held fixed.

Thus $m(q)=\lambda$, and all denominators in \eqref{sm:q} are
positive. Put
\[
 D(z)=(a-\gamma z_1)\prod_{i=2}^n(a+\beta z_i),\qquad
 T_j(z)=\sigma_j(z_2,\ldots,z_n).
\]

The strict choices $\gamma>1$ and $\beta<1$ provide a further essential
feature. If one instead took $\gamma=\beta=1$, the chain-rule calculation
would reproduce the diagonal part exactly: $-\frac12D^2\mathcal F(q)[v,v]=L_\lambda(y)$.
For the parameters in \eqref{sm:parameters}, the derivative ratios computed
below give the more useful identity
\[
\begin{aligned}
 L_\lambda(y)={}&-\frac12D^2\mathcal F(q)[v,v]
 +\frac{\gamma-1}{2(\gamma+1)}\frac{F^1}{a}y_1^2\\
 &+\sum_{i=2}^n\frac{a(1-\beta)F^i}{(a-\lambda_i)(a-\beta\lambda_i)}y_i^2,
 \qquad v_i=\frac{y_i}{m_i'(q_i)}.
\end{aligned}
\]
Every displayed residual coefficient is strictly positive. The first vanishes
when $\gamma=1$, and all the others vanish when $\beta=1$. These residual
coefficients are precisely the $A_i$ defined in \eqref{sm:good-coefficients}.
They are needed to absorb the $f$-dependent error arising from the denominator
factors in the logarithmic-concavity calculation and to obtain the quantitative
tangential estimate used in the completion of the affine estimate. The additional
inequality $\beta>\gamma(n-k)/k$ ensures positivity of the coefficients of the
cleared polynomial, as verified below.

\subsubsection{The polynomial and its hyperbolicity cone}

Clearing denominators gives a polynomial $\Phi$ satisfying
\begin{equation}\label{sm:clearing}
 D(z)\sigma_k(m(z))=a^k\Phi(z),
\end{equation}
where
\begin{equation}\label{sm:polynomial}
\begin{aligned}
 \Phi(z)
 ={}&\sum_{\ell=k}^{n-1}
 \binom\ell k a^{n-\ell}\beta^{\ell-k}T_\ell
 +z_1a^{n-k}T_{k-1}\\
 &+z_1\sum_{\ell=k}^{n-1}
 a^{n-1-\ell}\beta^{\ell-k}
 \left[\beta\binom\ell{k-1}-\gamma\binom\ell k\right]T_\ell.
\end{aligned}
\end{equation}
To verify this expression, expand first according to the occurrence of
$m_1$, and use
\[
 \prod_{i=2}^n(a+\beta z_i)\,
 \sigma_j(m_2,\ldots,m_n)
 =\sum_{\ell=j}^{n-1}
 \binom\ell j a^{n-1+j-\ell}\beta^{\ell-j}T_\ell.
\]
The binomial coefficient counts the chosen $j$ indices inside a support
of size $\ell$. All coefficients displayed in \eqref{sm:polynomial}
are positive, since
\[
 \frac{\binom\ell k}{\binom\ell{k-1}}
 =\frac{\ell-k+1}{k}
 \leq\frac{n-k}{k}<\frac\beta\gamma.
\]
The degree is $n$, with highest homogeneous part
\begin{equation}\label{sm:top-degree}
 \Phi^{[n]}(z)=c_{\gamma,\beta}\prod_{i=1}^n z_i,
 \qquad
 c_{\gamma,\beta}
 =\beta^{n-1-k}\binom{n-1}{k-1}
 \left(\beta-\gamma\frac{n-k}{k}\right)>0.
\end{equation}
Both types of maps in \eqref{sm:maps} preserve the upper half-plane:
\[
 \operatorname{Im}m_1(z)
 =\frac{a^2\operatorname{Im}z}{|a-\gamma z|^2}>0,
 \qquad
 \operatorname{Im}m_i(z)
 =\frac{a^2\operatorname{Im}z}{|a+\beta z|^2}>0.
\]
Lemma~\ref{poly:stability} and \eqref{sm:clearing} therefore imply
that $\Phi$ is real stable. Lemma~\ref{poly:log-concavity} requires a homogeneous
polynomial, whereas $\Phi$ is not homogeneous in $z$ with $a$ fixed.
We therefore homogenize $\Phi$. Its polynomial homogenization is
\[
\begin{aligned}
 \widehat\Phi(\tau,z)&=\tau^n\Phi(z/\tau)\\
 &=\sum_{\ell=k}^{n-1}\binom\ell k
 (\tau a)^{n-\ell}\beta^{\ell-k}T_\ell
 +z_1(\tau a)^{n-k}T_{k-1}\\
 &\quad+z_1\sum_{\ell=k}^{n-1}
 (\tau a)^{n-1-\ell}\beta^{\ell-k}
 \left[\beta\binom\ell{k-1}-\gamma\binom\ell k\right]T_\ell.
\end{aligned}
\]
Here $T_\ell=\sigma_\ell(z_2,\ldots,z_n)$. The explicit expression is obtained
from \eqref{sm:polynomial} by replacing $a$ with $\tau a$; it is homogeneous
of degree $n$ in $(\tau,z)$ and defines $\widehat\Phi$ also at $\tau=0$.

This polynomial is hyperbolic in $e=(0,1,\ldots,1)$. Indeed, for real $\tau_0\ne0$
and real $z$, stability and conjugation give real-rootedness of
$t\mapsto\widehat\Phi(\tau_0,z+t\mathbf1)$. The same real-rootedness follows
from the hyperbolicity of $\sigma_k$ in $\mathbf1=(1,\ldots,1)\in\mathbb R^n$,
together with Lemma~\ref{poly:cone} and the half-plane-preserving maps above.
More explicitly, for nonreal $t$ the coordinates of $(z+t\mathbf1)/\tau_0$
all have imaginary part $\operatorname{Im}t/\tau_0$, so they lie together
in the open upper or open lower half-plane. Real stability of $\Phi$ and
its real coefficients imply that $\tau_0^n\Phi((z+t\mathbf1)/\tau_0)\ne0$.
At $\tau_0=0$ this polynomial is $c_{\gamma,\beta}\prod_i(z_i+t)$.
In every case its leading coefficient is $c_{\gamma,\beta}>0$, and
$\widehat\Phi(e)=c_{\gamma,\beta}>0$.

Let $\mathcal C$ be the component of $\{\widehat\Phi>0\}$ containing
$e$. By Lemma~\ref{poly:cone}, it is an open convex cone, and every
point in $\mathcal C$ is a hyperbolicity direction. The positive
orthant is contained in $\mathcal C$. To see this, first connect $e$
to $(1,\mathbf1)$ along $(t,\mathbf1)$, $0\leq t\leq1$, where
positivity follows from the coefficients. Then use the connected
positive orthant, on which $\widehat\Phi>0$.

\subsubsection{The transformed point belongs to the cone}

We verify that $(1,q)\in\mathcal C$. Choose
$\lambda_2<b<a$, keep the first coordinate equal to $a$,
and replace each remaining coordinate by
\[
 \lambda_i(t)=(1-t)\lambda_i+tb,
 \qquad0\leq t\leq1,\quad i\geq2.
\]
Set $\delta_1(t)=0$ and $\delta_i(t)=t(b-\lambda_i)\geq0$.
Multiaffinity and deletion positivity at the given $\lambda$ imply
\[
 \sigma_k(\lambda(t))
 =f+\sum_{r=1}^k\sum_{|I|=r}
 \sigma_{k-r}(\lambda\mid I)\prod_{i\in I}\delta_i(t)
 \geq f>0.
\]
Let $q(t)=m^{-1}(\lambda(t))$, using \eqref{sm:q}. Every
$\lambda_i(t)<a$ for $i\geq2$, so the inverse maps are continuous
along the entire path. Moreover,
\[
 a-\gamma q_1(t)=\frac{a}{1+\gamma}>0,
 \qquad
 a+\beta q_i(t)=\frac{a^2}{a-\beta\lambda_i(t)}>0.
\]
Thus $D(q(t))>0$, and \eqref{sm:clearing} gives
$\widehat\Phi(1,q(t))>0$ along this path. At $t=1$, all coordinates
of $(1,q(1))$ are positive, so this endpoint belongs to $\mathcal C$.
The path lies in the same positive component, proving that
$(1,q)\in\mathcal C$.

\subsubsection{Derivative ratios of the coordinate maps}

Direct differentiation of \eqref{sm:maps}, evaluated at
\eqref{sm:q}, gives
\begin{equation}\label{sm:ratios}
 \frac{m_1''(q_1)}{(m_1'(q_1))^2}
 =\frac{2\gamma}{a(1+\gamma)},\qquad
 \frac{m_i''(q_i)}{(m_i'(q_i))^2}
 =-\frac{2\beta}{a-\beta\lambda_i}\quad(i\geq2).
\end{equation}
For example,
\[
\begin{aligned}
 m_1'(z)&=\frac{a^2}{(a-\gamma z)^2},
 &m_1''(z)&=\frac{2\gamma a^2}{(a-\gamma z)^3},\\
 m_i'(z)&=\frac{a^2}{(a+\beta z)^2},
 &m_i''(z)&=-\frac{2\beta a^2}{(a+\beta z)^3}.
\end{aligned}
\]
In particular, every $m_i'(q_i)$ is strictly positive.

\subsubsection{Logarithmic concavity and the tangential bound}

Since $(1,q)\in\mathcal C$, Lemma~\ref{poly:log-concavity} gives
\begin{equation}\label{sm:log-concavity}
 D^2\log\Phi(q)[v,v]
 =D^2\log\widehat\Phi(1,q)[(0,v),(0,v)]\leq0.
\end{equation}
The equality uses $\widehat\Phi(1,z)=\Phi(z)$, with the homogenizing variable
held fixed. The inequality holds for every real $v$. For $\sum_iF^iy_i=0$,
set $v_i=y_i/m_i'(q_i)$ and keep this direction fixed when differentiating.
For $\mathcal F=\sigma_k\circ m$, we have
\[
 \mathcal F(q)=f,\qquad
 D\mathcal F(q)[v]=\sum_iF^im_i'(q_i)v_i=\sum_iF^iy_i=0.
\]
Consequently, the first-derivative square in the logarithmic Hessian vanishes,
and the chain rule gives
\[
\begin{aligned}
 D^2\log\mathcal F(q)[v,v]
 &=\frac{D^2\mathcal F(q)[v,v]}{f}
 -\frac{\bigl(D\mathcal F(q)[v]\bigr)^2}{f^2}\\
 &=\frac1f\left(y^{\mathsf T}Hy
 +\sum_{i=1}^nF^i\frac{m_i''(q_i)}{(m_i'(q_i))^2}y_i^2\right).
\end{aligned}
\]
On a neighborhood of $q$, all factors of $D$ and $\mathcal F$ are positive, and
\[
 \log\Phi=\log D+\log\mathcal F-k\log a.
\]
Write
\begin{equation}\label{sm:error-coordinates}
 \eta_1=-\frac{\gamma y_1}{(1+\gamma)a},\qquad
 \eta_i=\frac{\beta y_i}{a-\beta\lambda_i}\quad(i\geq2).
\end{equation}
Using $v_i=y_i/m_i'(q_i)$ gives $\eta_1=-\gamma v_1/(a-\gamma q_1)$ and
$\eta_i=\beta v_i/(a+\beta q_i)$ for $i\geq2$. Since $\log D$ is a sum
of logarithms of one-variable linear factors,
\[
 D^2\log D(q)[v,v]
 =-\frac{\gamma^2v_1^2}{(a-\gamma q_1)^2}
 -\sum_{i=2}^n\frac{\beta^2v_i^2}{(a+\beta q_i)^2}
 =-\sum_{i=1}^n\eta_i^2.
\]
As $a$ is fixed, combining these identities with \eqref{sm:log-concavity}
yields $0\geq-\sum_i\eta_i^2+D^2\mathcal F(q)[v,v]/f$.
Multiplying by $f>0$ gives
\[
 y^{\mathsf T}Hy
 +\sum_{i=1}^nF^i\frac{m_i''(q_i)}{(m_i'(q_i))^2}y_i^2
 \leq f\sum_{i=1}^n\eta_i^2.
\]
Equivalently,
\begin{equation}\label{sm:good-error}
 L_\lambda(y)\geq A_1y_1^2+\sum_{i=2}^n A_iy_i^2
 -\frac f2\sum_{i=1}^n\eta_i^2,
\end{equation}
where
\begin{equation}\label{sm:good-coefficients}
 A_1=\frac{\gamma-1}{2(\gamma+1)}\frac{F^1}{a},
 \qquad
 A_i=\frac{a(1-\beta)F^i}{(a-\lambda_i)(a-\beta\lambda_i)}
 \quad(i\geq2).
\end{equation}
The first-coordinate error satisfies the exact identity
\[
 \frac f2\eta_1^2
 =\frac{f\gamma^2}{(\gamma^2-1)aF^1}A_1y_1^2.
\]
For $i\geq2$, use $F^i\geq F^1$ and
$\beta(a-\lambda_i)\leq a-\beta\lambda_i$ to obtain
\[
 \frac f2\eta_i^2
 \leq\frac{f\beta}{2a(1-\beta)F^1}A_iy_i^2.
\]
Our parameters satisfy
\begin{equation}\label{sm:epsilon-parameters}
 \varepsilon_I=\frac{1-\beta}{\beta}
 \leq\frac{\gamma^2-1}{2\gamma^2}.
\end{equation}
For completeness, the second inequality is equivalent to
$2n^2\leq(2k+n)(3n-2k)$, whose right side minus left side equals
\[
 n^2+4kn-4k^2=k^2+6k(n-k)+(n-k)^2>0.
\]
Condition \eqref{sm:regime} therefore makes each error no greater than
one half of its corresponding term in \eqref{sm:good-coefficients}.
It follows that
\[
 L_\lambda(y)\geq\frac12\sum_{i=1}^n A_i y_i^2.
\]
To estimate these coefficients uniformly, \eqref{sm:coordinate-bounds}
gives
\[
 0<a-\lambda_i<na,\qquad
 0<a-\beta\lambda_i<na\quad(i\geq2).
\]
Also,
\[
 \frac{1-\beta}{n^2}=\theta,\qquad
 \frac{\gamma-1}{2(\gamma+1)}
 =\frac{\Delta}{2(3n-2k)}\geq\theta.
\]
The last inequality is equivalent to $2kn^2\geq3n-2k$, which holds
for the stated $n$ and $k$. Consequently,
\[
 A_i\geq\frac\theta a F^i\quad(1\leq i\leq n).
\]
Returning to the notation $\lambda_1$, we have proved the tangential
estimate
\begin{equation}\label{sm:tangential}
 L_\lambda(y)\geq\frac{\theta}{2\lambda_1}
 \sum_{i=1}^nF^iy_i^2,
 \qquad \sum_{i=1}^nF^iy_i=0.
\end{equation}

\subsection{Completion of the affine estimate}

\begin{proof}[Proof of Proposition~\ref{sm:main}]
For arbitrary $x$, define
\begin{equation}\label{sm:shift-definitions}
 y=x-\frac{s}{F^1}e_1,\qquad
 Q(y)=\sum_{i=1}^n F^i y_i^2,\qquad
 T=\sum_{i=2}^n G_i y_i,
\end{equation}
where $e_1=(1,0,\ldots,0)$. In components,
\[
 (y_1,\ldots,y_n)=\left(x_1-\frac{s}{F^1},x_2,\ldots,x_n\right).
\]
Then $\sum_iF^iy_i=0$. Direct expansion of \eqref{sm:quadratic}, using
$F^i=\lambda_1U_i+G_i$, gives the exact identity
\begin{equation}\label{sm:shift-identity}
 L_\lambda(x)=L_\lambda(y)
 +\frac{sT}{\lambda_1F^1}
 -\frac{s^2}{2\lambda_1F^1}.
\end{equation}
By Lemma~\ref{sm:mixed} and weighted Cauchy's inequality,
\begin{equation}\label{sm:T-bound}
 T^2\leq
 \left(\sum_{i=2}^n\frac{G_i^2}{F^i}\right)
 \left(\sum_{i=2}^n F^i y_i^2\right)
 \leq\mathfrak h F^1Q(y).
\end{equation}

The tangential estimate \eqref{sm:tangential} implies
\[
 T^2\leq\frac{2\mathfrak h}{\theta}
 \lambda_1F^1L_\lambda(y).
\]
The exact shift identity \eqref{sm:shift-identity} and completion of
a square give
\begin{align*}
 L_\lambda(x)
 &\geq L_\lambda(y)
 -|s|\sqrt{\frac{2\mathfrak h}{\theta\lambda_1F^1}}
   \sqrt{L_\lambda(y)}
 -\frac{s^2}{2\lambda_1F^1}\\
 &\geq-\left(\frac12+\frac{\mathfrak h}{2\theta}\right)
 \frac{s^2}{\lambda_1F^1}\\
 &=-\left(\frac12+\frac{2k^2n^2(n-k)}{2k-n}\right)
 \frac{s^2}{\lambda_1F^1}.
\end{align*}
This proves \eqref{sm:main-bound}. Its stated consequences follow from
$\lambda_1F^1\geq c_0$ and $s=C_1\lambda_1$.
\end{proof}

\section{The complementary regime}\label{sec:large-f}

We now consider the alternative
\begin{equation}\label{lg:regime}
 f=\sigma_k(\lambda)\geq \varepsilon_I\lambda_1F^1,
 \qquad 0<M_1\leq f\leq M_2,
\end{equation}
where $\varepsilon_I$ is the constant fixed in \eqref{sm:constants}.
Throughout this section, $3\leq k<n\leq 2k-1$ and
\[
 \lambda\in\Gamma_k,
 \qquad \lambda_1>\lambda_2\geq\cdots\geq\lambda_n.
\]
Set $H=D^2\sigma_k(\lambda)$. Thus
\[
 H_{ii}=0,
 \qquad H_{ij}=\sigma_{k-2}(\lambda\mid ij)\quad(i\ne j).
\]
For $x\in\mathbb R^n$, we use the notation
\begin{equation}\label{lg:quadratic-form}
 \begin{split}
 L_\lambda(x)
 &=-\frac12H[x,x]
   +\sum_{i=2}^n\frac{F^i}{\lambda_1-\lambda_i}x_i^2
   -\frac{F^1}{2\lambda_1}x_1^2,
 \\
 s&=\sum_{i=1}^nF^ix_i.
 \end{split}
\end{equation}
The estimate in this regime follows from a spectral lower bound and a
concavity inequality of Zhang. We state the latter on the level set
$\sigma_k=1$ so that its constants remain uniform as $f$ varies.

\begin{lemma}[Zhang~\cite{Zhang}, Lemma~1.1]\label{lg:zhang}
Let $3\leq k<n$ and $A>0$. There exist constants $\delta>0$, $K>0$ and
$R_{\mathrm Z}>0$, depending only on $n,k,A$, with the following property. If
\[
 \mu\in\Gamma_k,\qquad \sigma_k(\mu)=1,\qquad
 \mu_1\geq\cdots\geq\mu_n>-A,\qquad \mu_1\geq R_{\mathrm Z},
\]
then, for every $\xi\in\mathbb R^n$,
\begin{equation}\label{lg:zhang-inequality}
 \begin{split}
 &-D^2\sigma_k(\mu)[\xi,\xi]
 +K\left(\sum_{i=1}^n\sigma_{k-1}(\mu\mid i)\xi_i\right)^2
 +2\sum_{i=2}^n
     \frac{\sigma_{k-1}(\mu\mid i)}{\mu_1+A+1}\xi_i^2
 \\
 &\hspace{35mm}\geq
 (1+\delta)\frac{\sigma_{k-1}(\mu\mid1)}{\mu_1}\xi_1^2.
 \end{split}
\end{equation}
\end{lemma}

This is the normalized form of \cite[inequality~(1.10)]{Zhang}; we use
it without repeating its proof. We also use deletion positivity and
the Newton--Maclaurin inequalities for elementary symmetric polynomials.

\begin{lemma}\label{lg:spectral-bound}
For every ordered $\lambda\in\Gamma_k$ as above,
\begin{equation}\label{lg:negative-eigenvalue}
 (-\lambda_n)_+^k\leq C^{\mathrm{spec}}_{n,k}\lambda_1F^1,
 \qquad
 C^{\mathrm{spec}}_{n,k}=(n-k)
 \frac{\binom{n-2}{k-1}^{k-2}}
      {\binom{n-2}{k-2}^{k-1}}.
\end{equation}
Consequently, under \eqref{lg:regime},
\begin{equation}\label{lg:B}
 \lambda_n\geq-B,
 \qquad
 B=\left(\frac{C^{\mathrm{spec}}_{n,k}M_2}{\varepsilon_I}\right)^{1/k}.
\end{equation}
\end{lemma}

\begin{proof}
It suffices to consider $\lambda_n=-b<0$. Put
\[
 \begin{gathered}
 w=(\lambda_2,\ldots,\lambda_{n-1}),\qquad
 T=\sigma_{k-2}(w),\qquad U=\sigma_{k-1}(w),
 \\
 V=\sigma_k(w),\qquad p=F^1.
 \end{gathered}
\]
Deletion positivity gives $(w,-b)\in\Gamma_{k-1}$. The recurrence
\[
 \sigma_j(w)=\sigma_j(w,-b)+b\sigma_{j-1}(w)
\]
then shows that $w\in\Gamma_{k-1}$. In particular, $T>0$, and
\begin{equation}\label{lg:recurrence}
 p=U-bT>0,
 \qquad f=\lambda_1p+V-bU.
\end{equation}
By Maclaurin's inequality and $U>bT$,
\[
 bT<U\leq
 \binom{n-2}{k-1}
 \left(\frac{T}{\binom{n-2}{k-2}}\right)^{(k-1)/(k-2)},
\]
and hence
\begin{equation}\label{lg:T-lower-bound}
 T>d_{n,k}b^{k-2},
 \qquad
 d_{n,k}=
 \frac{\binom{n-2}{k-2}^{k-1}}
      {\binom{n-2}{k-1}^{k-2}}.
\end{equation}
Euler's identity, positivity of the first derivatives of $\sigma_{k-1}(w)$,
and $w_i\leq\lambda_1$ give
\begin{equation}\label{lg:Euler}
 (k-1)U
 =\sum_{i=1}^{n-2}w_i\sigma_{k-2}(w\mid i)
 \leq\lambda_1(n-k)T.
\end{equation}
Newton's inequality yields
\[
 V\leq\vartheta\frac{U^2}{T},
 \qquad
 \vartheta=\frac{(k-1)(n-k-1)}{k(n-k)},
 \qquad
 1-\vartheta=\frac{n-1}{k(n-k)}.
\]
When $n=k+1$, the same formula holds with $V=\vartheta=0$.
Since $U=p+bT$, \eqref{lg:recurrence} and \eqref{lg:Euler} imply
\begin{align*}
 0<f
 &\leq\lambda_1p+\vartheta\frac{U^2}{T}-bU
 =\lambda_1p+\vartheta\frac{U}{T}p-(1-\vartheta)bU
 \\
 &\leq\frac{n-1}{k}\lambda_1p
       -\frac{n-1}{k(n-k)}b^2T.
\end{align*}
It follows that $b^2T<(n-k)\lambda_1p$.
Combining this with \eqref{lg:T-lower-bound} proves
\eqref{lg:negative-eigenvalue}. Finally, \eqref{lg:regime} gives
\[
 \lambda_1F^1\leq\frac{f}{\varepsilon_I}
 \leq\frac{M_2}{\varepsilon_I},
\]
which yields \eqref{lg:B}.
\end{proof}

\begin{proposition}\label{lg:affine-estimate}\label{lg:main}
Assume \eqref{lg:regime}. There exist $\Lambda>0$, $K>0$ and $\delta>0$,
depending only on $n,k,M_1,M_2$, such that, whenever
$\lambda_1\geq\Lambda$,
\begin{equation}\label{lg:all-vectors}
 L_\lambda(x)
 \geq\frac{\delta}{2}\frac{F^1}{\lambda_1}x_1^2
      -\frac{K}{2f}\left(\sum_{i=1}^nF^ix_i\right)^2
 \qquad\text{for every }x\in\mathbb R^n.
\end{equation}
In particular, for each fixed $C_1\in\mathbb R$, the affine constraint
$s=C_1\lambda_1$ implies
\begin{equation}\label{lg:affine-conclusion}
 L_\lambda(x)\geq-C_2\lambda_1^2,
 \qquad C_2=1+\frac{KC_1^2}{2M_1},
\end{equation}
where $C_2$ depends only on $n,k,C_1,M_1,M_2$.
\end{proposition}

\begin{proof}
Let $B$ be as in \eqref{lg:B}, and set
\[
 t=f^{1/k},\qquad \mu=\lambda/t,
 \qquad A=1+\frac{B}{M_1^{1/k}}.
\]
Then $\sigma_k(\mu)=1$ and
\[
 \mu_n\geq-\frac{B}{t}\geq-\frac{B}{M_1^{1/k}}>-A.
\]
Let $R_{\mathrm Z},K,\delta$ be the constants in Lemma~\ref{lg:zhang} for this
fixed $A$. Taking
\[
 \Lambda=M_2^{1/k}\max\{R_{\mathrm Z},1\}
\]
ensures $\mu_1=\lambda_1/t\geq R_{\mathrm Z}$ whenever
$\lambda_1\geq\Lambda$.

Apply \eqref{lg:zhang-inequality} with $\xi=x$ and multiply by
$t^{k-2}$. Since
\[
 \sigma_{k-1}(\mu\mid i)=t^{-(k-1)}F^i,
 \qquad D^2\sigma_k(\mu)=t^{-(k-2)}H,
 \qquad t^k=f,
\]
we obtain
\begin{equation}\label{lg:scaled-inequality}
 -H[x,x]+\frac{K}{f}s^2
 +2\sum_{i=2}^n\frac{F^i}{\lambda_1+(A+1)t}x_i^2
 \geq(1+\delta)\frac{F^1}{\lambda_1}x_1^2.
\end{equation}
Moreover, $(A+1)t\geq B+2M_1^{1/k}>B$, whereas
\[
 0<\lambda_1-\lambda_i\leq\lambda_1+B
 <\lambda_1+(A+1)t\qquad(i\geq2).
\]
Since $F^i>0$, replacing the denominators in the positive sum in
\eqref{lg:scaled-inequality} by $\lambda_1-\lambda_i$ increases its
left-hand side. Subtracting $(F^1/\lambda_1)x_1^2$ and dividing by $2$
therefore proves \eqref{lg:all-vectors}. Finally, $f\geq M_1$ and
$s=C_1\lambda_1$ give
\[
 L_\lambda(x)\geq-\frac{KC_1^2}{2M_1}\lambda_1^2,
\]
which implies \eqref{lg:affine-conclusion}.
\end{proof}

\begin{corollary}[The convex case]\label{lg:convex}
Suppose instead that $\lambda_n\geq0$, while the ordering, cone
condition and $0<M_1\leq\sigma_k(\lambda)\leq M_2$ are retained.
Then the conclusions of Proposition~\ref{lg:affine-estimate} hold
without separately imposing
$f\geq\varepsilon_I\lambda_1F^1$.
\end{corollary}

\begin{proof}
All coordinates are nonnegative, so
\[
 f=\lambda_1F^1+\sigma_k(\lambda\mid1)
 \geq\lambda_1F^1>\varepsilon_I\lambda_1F^1,
\]
where $0<\varepsilon_I<1$ and $F^1>0$. Thus the required alternative
is automatic. Equivalently, one may apply Lemma~\ref{lg:zhang} directly
with $A=1$ to $\mu=\lambda/f^{1/k}$ and repeat the preceding scaling
argument; Lemma~\ref{lg:spectral-bound} is unnecessary in this case.
\end{proof}

\begin{remark}\label{lg:threshold-remark}
The role of the inequality
$f\geq\varepsilon_I\lambda_1F^1$ is precisely to supply the uniform
spectral lower bound needed in Lemma~\ref{lg:zhang}.
The same argument applies with $\varepsilon_I$ replaced by any fixed
$\varepsilon>0$, with the constants then also depending on
$\varepsilon$.
\end{remark}

\section{Repeated largest principal curvatures and completion of the estimate}
\label{sec:completion}

\subsection{The repeated-eigenvalue case}

We return to the maximum point $x_0$ of the geometric test function
\eqref{eq:test-function}. Suppose that the largest principal curvature has
multiplicity $m\ge2$. The upper-contact identities
\eqref{eq:contact-first}, together with Codazzi, give
\begin{equation}\label{eq:repeated-zero}
 h_{111}=h_{221}=h_{212}=0,
 \qquad h_{ii1}=h_{111}=0\quad(1\le i\le m).
\end{equation}
Here $h_{111}=h_{221}$ and $h_{212}=0$ follow from
\eqref{eq:contact-first}, whereas $h_{221}=h_{212}$ is Codazzi. Thus the
vector $x_i=h_{ii1}$ arising in \eqref{eq:L-reduction} satisfies
\[
 x_1=\cdots=x_m=0.
\]
This vanishing uses the smooth upper contact at the maximum point.

For completeness, the quadratic-form estimate needed under this condition
follows directly from the standard concavity of $\sigma_k^{1/k}$.

\begin{proposition}\label{prop:repeated-bound}
Let $\la\in\Gamma_k$ satisfy
$\la_1=\cdots=\la_m>\la_{m+1}\ge\cdots\ge\la_n$, with $2\le m\le n$,
where the comparison with $\la_{m+1}$ is omitted when $m=n$.
Let $L_\la$ be defined by \eqref{eq:L-definition}. If
$x_1=\cdots=x_m=0$ and $s=\sum_iF^ix_i$, then
\begin{equation}\label{eq:repeated-affine-bound}
 L_\la(x)\ge-\frac{k-1}{2kf}s^2,
 \qquad f=\sigma_k(\la).
\end{equation}
In particular, if $f\ge M_1>0$ and $s=C_1\la_1$ with $|C_1|\le B_1$, then
\begin{equation}\label{eq:repeated-final-bound}
 L_\la(x)\ge-\frac{k-1}{2kM_1}B_1^2\la_1^2.
\end{equation}
The same statement holds for $m=n$, with the usual empty-sum convention.
\end{proposition}

\begin{proof}
By \eqref{eq:sigma-concavity},
\[
 -\frac12D^2\sigma_k(\la)[x,x]
 \ge-\frac{k-1}{2kf}\left(\sumall F^ix_i\right)^2.
\]
Since $x_1=0$, the last term in \eqref{eq:L-definition} vanishes. Each term
in its middle sum is nonnegative, because $F^i>0$ and
$\la_1-\la_i>0$ for $i>m$. This proves
\eqref{eq:repeated-affine-bound}; \eqref{eq:repeated-final-bound} follows
from $f\ge M_1$ and the bound on $C_1$.
\end{proof}

\subsection{Completion of the proof of Theorem~\ref{thm:curvature}}

The constants in the preceding quadratic-form estimates can now be chosen
uniformly for the vector arising in the maximum-principle computation.
By \eqref{eq:affine-constraint} and $\la_1\ge1$, there exists a constant
$B_1$, depending only on the prescribed data, such that
\begin{equation}\label{eq:bounded-affine-parameter}
 s=\sumall F^ix_i=C_1\la_1,\qquad |C_1|\le B_1.
\end{equation}
The cone estimate \eqref{eq:cone-lower-bound} also gives
\begin{equation}\label{eq:c0-choice}
 \la_1F^1\ge\frac{k}{n}f\ge\frac{k}{n}M_1=:c_0>0.
\end{equation}
In particular, $c_0$ is a lower bound at the given vector $\la$.

First suppose that $m=1$. If
$0<f\le\varepsilon_I\la_1F^1$, Proposition~\ref{sm:main} and
\eqref{eq:bounded-affine-parameter}--\eqref{eq:c0-choice} imply
\begin{equation}\label{eq:small-f-geometric-bound}
 L_\la(x)\ge-C_{\mathrm{small}}\la_1^2,
 \qquad
 C_{\mathrm{small}}
 =\left(\frac12+\frac{2k^2n^2(n-k)}{2k-n}\right)\frac{B_1^2}{c_0}.
\end{equation}
If $f\ge\varepsilon_I\la_1F^1$, Proposition~\ref{lg:main} gives, whenever
$\la_1\ge\Lambda$,
\begin{equation}\label{eq:large-f-geometric-bound}
 L_\la(x)\ge-C_{\mathrm{large}}\la_1^2,
 \qquad C_{\mathrm{large}}=\frac{KB_1^2}{2M_1}.
\end{equation}
The threshold $\Lambda$ and the constant $K$ depend only on
$n,k,M_1,M_2$. If $\la_1<\Lambda$, the required estimate at $x_0$ is
already available.

When $m\ge2$, \eqref{eq:repeated-zero} and
Proposition~\ref{prop:repeated-bound} give the same type of bound, with
\[
 C_{\mathrm{rep}}=\frac{k-1}{2kM_1}B_1^2.
\]
Thus, with
\begin{equation}\label{eq:unified-C2}
 C_2=1+\max\{C_{\mathrm{small}},C_{\mathrm{large}},C_{\mathrm{rep}}\},
\end{equation}
we have, in every remaining case,
\begin{equation}\label{eq:unified-L-bound}
 L_\la(x)\ge-C_2\la_1^2.
\end{equation}
All constants in this inequality are independent of $N$.

Substitution into \eqref{eq:L-reduction} yields
\begin{equation}\label{eq:last-maximum-inequality}
 0\ge (N-1)\sumall F^i\la_i^2-(2C_2+C)\la_1-CN.
\end{equation}
Moreover, \eqref{eq:c0-choice} implies
\[
 \sumall F^i\la_i^2\ge F^1\la_1^2\ge c_0\la_1.
\]
Choose $N\ge1$, depending only on the prescribed data, so that
\[
 (N-1)c_0\ge2C_2+C+1.
\]
It follows from \eqref{eq:last-maximum-inequality} that
$\la_1(x_0)\le CN$. Together with the previously excluded bounded cases,
this gives a bound for $\la_1(x_0)$ depending only on the prescribed data.

Finally, by maximality of \eqref{eq:test-function}, for any $X\in M$,
\[
 \la_1(X)\le\la_1(x_0)
       \left(\frac{u(X)}{u(x_0)}\right)^N
 \le\left(\frac R{u_*}\right)^N\la_1(x_0).
\]
If the maximum lies on the boundary, the same calculation bounds
$\la_1(X)$ by $(R/u_*)^N\max_{\partial M}\la_1$.
Using $|\la_i|\le\la_1$ from \eqref{eq:positive-sum} now proves
\eqref{eq:curvature-estimate}.\hfill$\square$

\appendix
\section{Auxiliary polynomial and convexity facts}
\label{poly:section}

We record the polynomial facts used in Section~\ref{sm:section}.
A real polynomial is called real stable if it does not vanish when
each variable has strictly positive imaginary part.

\begin{lemma}[Real stability of elementary symmetric polynomials]
\label{poly:stability}
If $y_1,\ldots,y_n\in\mathbb C$ satisfy
$\operatorname{Im}y_i>0$, then
$\sigma_j(y_1,\ldots,y_n)\ne0$ for $1\leq j\leq n$.
\end{lemma}

\begin{proof}
Let $p_0(s)=\prod_i(s+y_i)$. Its roots lie in the open lower
half-plane. If a nonconstant polynomial $p$ has roots $\rho_j$ there
(open lower half-plane), counted with multiplicity, then, for $\operatorname{Im}s\geq0$,
\[
 \operatorname{Im}\frac{p'(s)}{p(s)}
 =\operatorname{Im}\sum_j\frac1{s-\rho_j}
 =-\sum_j\frac{\operatorname{Im}s-\operatorname{Im}\rho_j}{|s-\rho_j|^2}<0.
\]
Thus $p'$ has no zero in the closed upper half-plane. Iteration gives
$p_0^{(n-j)}(0)\ne0$. Since
$p_0^{(n-j)}(0)=(n-j)!\sigma_j(y)$, the assertion follows.
\end{proof}

A real homogeneous polynomial $P$ of degree $d$ is hyperbolic in
direction $e$ if $P(e)\ne0$ and $t\mapsto P(u+te)$ has only real
roots for every real $u$.

\begin{lemma}[G\aa rding~{\cite[Section~2, Theorem~2]{Garding}}]\label{poly:cone}
Suppose that $P$ is hyperbolic in $e$, with $P(e)>0$. The connected
component $\mathcal K$ of $\{P\ne0\}$ containing $e$ is an open
convex cone. Moreover, $P$ is hyperbolic in every direction
$v\in\mathcal K$.
\end{lemma}

\begin{lemma}[Logarithmic concavity in the cone]
\label{poly:log-concavity}
Under the assumptions of Lemma~\ref{poly:cone}, for
$Q\in\mathcal K$ and every real vector $V$,
\[
 D^2\log P(Q)[V,V]\leq0.
\]
\end{lemma}

\begin{proof}
By Lemma~\ref{poly:cone}, $Q$ is a hyperbolicity direction. Write
\[
 P(V+sQ)=P(Q)\prod_{\nu=1}^d(s+\rho_\nu),
 \qquad\rho_\nu\in\mathbb R.
\]
Homogeneity, first for $t\ne0$ and then by continuity, gives
\[
 \frac{P(Q+tV)}{P(Q)}=\prod_{\nu=1}^d(1+t\rho_\nu).
\]
Taking two derivatives at zero yields
\[
 D^2\log P(Q)[V,V]=-\sum_{\nu=1}^d\rho_\nu^2\leq0.
\]
This also covers a degree drop in $t$ through zero values of some
$\rho_\nu$.
\end{proof}

\clearpage
\begingroup
\fontsize{10}{11.4}\selectfont
\setlength{\parskip}{0pt}

\endgroup


\begin{thebibliography}{99}
\setlength{\itemsep}{1pt}
\setlength{\parsep}{0pt}
\setlength{\parskip}{0pt}
\interlinepenalty=10000
\bibitem{Alexandrov1942}
A.~D.~Alexandrov,
\emph{Existence and uniqueness of a convex surface with a given integral curvature},
C. R. (Doklady) Acad. Sci. URSS (N.S.) \textbf{35} (1942), 131--134.

\bibitem{Alexandrov1956}
A.~D.~Alexandrov,
\emph{Uniqueness theorems for surfaces in the large. I},
Vestnik Leningrad. Univ. \textbf{11} (1956), 5--17 (Russian);
English translation, Amer. Math. Soc. Transl. (2) \textbf{21} (1962), 341--354.

\bibitem{BK}
I.~Bakelman and B.~Kantor,
\emph{Existence of spherically homeomorphic hypersurfaces in Euclidean space with prescribed mean curvature},
in \emph{Geometry and Topology}, Vol.~1,
Leningrad, 1974, 3--10 (Russian).

\bibitem{BF}
K.~J.~B\"or\"oczky and F.~Fodor,
\emph{The $L_p$ dual Minkowski problem for $p>1$ and $q>0$},
J. Differential Equations \textbf{266} (2019), no.~12, 7980--8033.
\href{https://arxiv.org/abs/1802.00933}{arXiv:1802.00933}.

\bibitem{BCD}
S.~Brendle, K.~Choi, and P.~Daskalopoulos,
\emph{Asymptotic behavior of flows by powers of the Gaussian curvature},
Acta Math. \textbf{219} (2017), 1--16.
\href{https://arxiv.org/abs/1610.08933}{arXiv:1610.08933}.

\bibitem{CNSI}
L.~Caffarelli, L.~Nirenberg, and J.~Spruck,
\emph{The Dirichlet problem for nonlinear second order elliptic equations. I. Monge--Amp\`ere equations},
Comm. Pure Appl. Math. \textbf{37} (1984), 369--402.

\bibitem{CNS}
L.~Caffarelli, L.~Nirenberg, and J.~Spruck,
\emph{Nonlinear second order elliptic equations. IV. Starshaped compact Weingarten hypersurfaces},
in \emph{Current Topics in Partial Differential Equations},
Y.~Ohya, K.~Kasahara, and N.~Shimakura (eds.),
Kinokuniya, Tokyo, 1986, 1--26.

\bibitem{CY}
S.-Y.~Cheng and S.-T.~Yau,
\emph{On the regularity of the solution of the $n$-dimensional Minkowski problem},
Comm. Pure Appl. Math. \textbf{29} (1976), 495--516.

\bibitem{Chu}
J.~Chu,
\emph{A simple proof of curvature estimate for convex solution of $k$-Hessian equation},
Proc. Amer. Math. Soc. \textbf{149} (2021), no.~8, 3541--3552.
\href{https://arxiv.org/abs/2003.00634}{arXiv:2003.00634}.

\bibitem{Garding}
L.~G\aa rding,
\emph{An inequality for hyperbolic polynomials},
Journal of Mathematics and Mechanics \textbf{8} (1959), no.~6, 957--965.

\bibitem{GG}
B.~Guan and P.~Guan,
\emph{Convex hypersurfaces of prescribed curvatures},
Ann. of Math. (2) \textbf{156} (2002), no.~2, 655--673.

\bibitem{GLL}
P.~Guan, J.~Li, and Y.~Y.~Li,
\emph{Hypersurfaces of prescribed curvature measure},
Duke Math. J. \textbf{161} (2012), no.~10, 1927--1942.
\href{https://doi.org/10.1215/00127094-1645550}{doi:10.1215/00127094-1645550}.

\bibitem{GLM}
P.~Guan, C.-S.~Lin, and X.~Ma,
\emph{The existence of convex body with prescribed curvature measures},
Int. Math. Res. Not. IMRN \textbf{2009} (2009), no.~11, 1947--1975.

\bibitem{GRW}
P.~Guan, C.~Ren, and Z.~Wang,
\emph{Global $C^2$-estimates for convex solutions of curvature equations},
Comm. Pure Appl. Math. \textbf{68} (2015), 1287--1325.
\href{https://arxiv.org/abs/1304.7062}{arXiv:1304.7062}.

\bibitem{HI}
Y.~Hu and M.~N.~Ivaki,
\emph{Prescribed $L_p$ curvature problem},
Adv. Math. \textbf{442} (2024), Paper No.~109566, 15~pp.
\href{https://arxiv.org/abs/2312.02122}{arXiv:2312.02122}.

\bibitem{LuSemi}
S.~Lu,
\emph{Curvature estimates for semi-convex solutions of Hessian equations in hyperbolic space},
Calc. Var. Partial Differential Equations \textbf{62} (2023), no.~9, Paper No.~257, 23~pp.
\href{https://arxiv.org/abs/2302.12409}{arXiv:2302.12409}.

\bibitem{LT}
S.~Lu and Y.-L.~Tsai,
\emph{A simple proof of curvature estimates for the $n-1$ Hessian equation},
Proc. Amer. Math. Soc. \textbf{154} (2026), no.~2, 893--904.
\href{https://arxiv.org/abs/2412.00576}{arXiv:2412.00576}.

\bibitem{Nirenberg1953}
L.~Nirenberg,
\emph{The Weyl and Minkowski problems in differential geometry in the large},
Comm. Pure Appl. Math. \textbf{6} (1953), 337--394.

\bibitem{Pogorelov1953}
A.~V.~Pogorelov,
\emph{On existence of a convex surface with a given sum of the principal radii of curvature},
Uspekhi Mat. Nauk \textbf{8} (1953), no.~3(55), 127--130 (Russian).

\bibitem{Pogorelov1978}
A.~V.~Pogorelov,
\emph{The Minkowski Multidimensional Problem},
translated from the Russian by V.~Oliker,
Scripta Series in Mathematics,
V.~H.~Winston \& Sons, Washington, D.C.;
Halsted Press, John Wiley \& Sons, New York, 1978.

\bibitem{RWone}
C.~Ren and Z.~Wang,
\emph{On the curvature estimates for Hessian equations},
Amer. J. Math. \textbf{141} (2019), no.~5, 1281--1315.
\href{https://arxiv.org/abs/1602.06535}{arXiv:1602.06535}.

\bibitem{RWtwo}
C.~Ren and Z.~Wang,
\emph{The global curvature estimate for the $n-2$ Hessian equation},
Calc. Var. Partial Differential Equations \textbf{62} (2023), no.~9, Paper No.~239, 50~pp.
\href{https://arxiv.org/abs/2002.08702}{arXiv:2002.08702}.

\bibitem{SX}
J.~Spruck and L.~Xiao,
\emph{A note on star-shaped compact hypersurfaces with prescribed scalar curvature in space forms},
Rev. Mat. Iberoam. \textbf{33} (2017), no.~2, 547--554.
\href{https://doi.org/10.4171/RMI/948}{doi:10.4171/RMI/948}.

\bibitem{TW}
A.~Treibergs and S.~W.~Wei,
\emph{Embedded hyperspheres with prescribed mean curvature},
J. Differential Geom. \textbf{18} (1983), no.~3, 513--521.

\bibitem{Yan}
J.~Yan,
\emph{Global Curvature Estimates for $\sigma_k$ Curvature Equations with $k\ge n/2$},
\href{https://arxiv.org/abs/2608.25665}{arXiv:2608.25665} (2026).

\bibitem{Yang}
F.~Yang,
\emph{Prescribed curvature measure problem in hyperbolic space},
Comm. Pure Appl. Math. \textbf{77} (2024), 863--898.
\href{https://doi.org/10.1002/cpa.22160}{doi:10.1002/cpa.22160}.

\bibitem{Zhang}
R.~Zhang,
\emph{$C^2$ estimates for $k$-Hessian equations and a rigidity theorem},
Adv. Math. \textbf{480} (2025), Part~A, Paper No.~110488.
\href{https://arxiv.org/abs/2408.10781v2}{arXiv:2408.10781v2}.
\end{thebibliography}
\end{document}